\documentclass{birkjour}%
\usepackage{amsmath}
\usepackage{amsfonts}
\usepackage{amssymb}
\usepackage[square, numbers]{natbib}
\usepackage{graphicx}
\usepackage{hyperref}%
\newtheorem{theorem}{Theorem}

\newtheorem{lemma}[theorem]{Lemma}

\begin{document}
	
	\title[Decay rates for (2+1)-dimensional oscillatory integral operators]{Decay rates for (2+1)-dimensional oscillatory integral operators with homogeneous polynomial phases}
	\author{Jayden Lang}\address{American Heritage School\\Plantation, FL 33325\\USA}\email{jaydenlang88@gmail.com}
	\author{Wan Tang}\address{Department of Biostatistics and Data Science\\Tulane University\\New Orleans, LA 70112\\USA}\email{wtang1@tulane.edu}

	\begin{abstract}%

		Consider the oscillatory integral operators
		\begin{equation}
			T_{\lambda}f(y)=\int_{\mathbb{R}^{2}}e^{i\lambda S\left(  x_{1},x_{2}%
				,y\right)  }\Phi(x_{1},x_{2},y)f(x_{1},x_{2})dx_{1}dx_{2},\nonumber
		\end{equation}
		where $\Phi(x_{1},x_{2},y)\in C_{0}^{\infty}\left(  \mathbb{R}^{3}\right)  $,
		$S\left(  x_{1},x_{2},y\right)  \in C_{0}^{\infty}\left(  \mathbb{R}%
		^{3}\right)  $ is real valued, and $\lambda$ is a large real number. We prove
		that, if $S\left(  x_{1},x_{2},y\right)  =y^{n_{1}}h_{n-n_{1}}\left(
		x_{1},x_{2}\right)  +\cdots+y^{n_{s}}h_{n-n_{s}}\left(  x_{1},x_{2}\right)  $
		is a homogeneous polynomial of degree $n,$ where $0<n_{1}<n_{2}<\cdots
		<n_{s}<n$, and $h_{n-n_{1}}\left(  x_{1},x_{2}\right)  $ and $h_{n-n_{s}%
		}\left(  x_{1},x_{2}\right)  $ are non-degenerate in the sense that there are
		no multiple factors when they are factored into linear terms over complex
		numbers, then for $\delta=\max \left(  \frac{n}{3},\frac{n-n_{s}}{2}%
		,n_{1}\right)  >1,$ $\left \Vert T_{\lambda}\right \Vert _{L^{2}\rightarrow
			L^{2}}=O\left(  \lambda^{-1/\left(  2\delta \right)  }\right)  $, while in the
		endpoint case $\delta=1$ the bound becomes $\left \Vert T_{\lambda}\right \Vert
		_{L^{2}\rightarrow L^{2}}=O\left(  \lambda^{-1/2}\log \lambda \right)  $. The
		decay rate is sharp, up to a power of $\log \lambda$ when $\delta=1$. We
		further show that $\delta$ is exactly the modified Newton distance for the
		phase function, thus verifies the conjecture of \citet{Greenleaf07} in this case.%
		
	\end{abstract}%
	\maketitle
	
\section{Introduction}

In this paper, we consider the $(2+1)$-dimensional oscillatory integral
operators with homogeneous polynomial phase functions. In general, a
$(m+n)$-dimensional oscillatory integral operators is defined by
\begin{equation}
T_{\lambda}f(\mathbf{y})=\int_{\mathbb{R}^{m}}e^{i\lambda S\left(
\mathbf{x},\mathbf{y}\right)  }\Phi(\mathbf{x},\mathbf{y})f(\mathbf{x}%
)\mathrm{d}\mathbf{x}, \label{ggg}%
\end{equation}
where $\mathbf{x}\in \mathbb{R}^{m}$ and $\mathbf{y}\in \mathbb{R}^{n},$ the
amplitude $\Phi \left(  \mathbf{x},\mathbf{y}\right)  \in C_{0}^{\infty}\left(
\mathbb{R}^{m+n}\right)  $, the phase function $S\left(  \mathbf{x}%
,\mathbf{y}\right)  \in C_{0}^{\infty}\left(  \mathbb{R}^{m+n}\right)  $ is
real valued, and $\lambda$ is a large real number. Understanding the L$^{2}%
$-operator norm decay of such oscillatory integral operators as $\lambda$ goes
to infinity is a central problem in harmonic analysis, with deep connections
to the smoothing properties of Fourier integral operators and generalized
Radon transforms. The decay rate is closely related to the singularity of the
phase function. At a nondegenerate point where the mixed Hessian matrix has
full rank, i.e., rank$\left(  \frac{\partial^{2}S\left(  \mathbf{x}%
,\mathbf{y}\right)  }{\partial \mathbf{x}\partial \mathbf{y}^{T}}\right)
=\min \left(  m,n\right)  ,$ one has $\left \Vert T_{\lambda}\right \Vert
_{L^{2}\left(  \mathbb{R}^{m}\right)  \rightarrow L^{2}\left(  \mathbb{R}%
^{n}\right)  }= O(\lambda^{-\frac{1}{2}\min \left(  m,n\right)  }),$ provided
that the amplitude function is supported in a sufficiently small neighborhood
of the point \citep{hormander73}. The situation becomes more complicated when
the point is degenerate.

In the ($1+1$)-dimensional case, the question has been completely settled by
\citet{Rychkov01} and \citet{Greenblatt05} for smooth phase functions.
However, the question in higher-dimensional settings is far less understood,
even for homogeneous polynomial phase functions. \citet{Fu99} considered some
(2+$m$)-dimensional cases with $m\geq2$, where the phase is given by $S\left(
\mathbf{x},\mathbf{y}\right)  =x_{1}\phi_{1}\left(  \mathbf{y}\right)
+x_{2}\phi_{2}\left(  \mathbf{y}\right)  ,$ and both $\phi_{1}\left(
\mathbf{y}\right)  $ and $\phi_{2}\left(  \mathbf{y}\right)  $ homogenous
polynomials of the same degree in $\mathbb{R}^{m}$. \citet{Greenleaf07}
further considered homogeneous phase functions in the case $\min \left(
m,n\right)  \geq2,$ under the \textquotedblleft rank one\textquotedblright%
\ condition that rank$\left(  \frac{\partial^{2}S\left(  \mathbf{x}%
,\mathbf{y}\right)  }{\partial \mathbf{x}\partial \mathbf{y}^{T}}\right)  \geq1$
for all $(\mathbf{x},\mathbf{y})\neq \mathbf{0}$.

In the (2+1)-dimensional cases, where $S\left(  x_{1},x_{2},y\right)
=y^{n_{1}}h_{n-n_{1}}\left(  x_{1},x_{2}\right)  +\cdots+y^{n_{s}}h_{n-n_{s}%
}\left(  x_{1},x_{2}\right)  ,$ $h_{m}\left(  x_{1},x_{2}\right)  $ is a
homogeneous polynomial of degree $m$ in $x_{1}$ and $x_{2}$, $0<n_{1}%
<n_{2}<\cdots<n_{s}<n$, \citet{tang06} obtained the optimal decay rate when
$n_{1}\leq \frac{n}{3}$, $n_{s}\geq \frac{n}{3},$ and both $h_{n-n_{1}}\left(
x_{1},x_{2}\right)  $ and $h_{n-n_{s}}\left(  x_{1},x_{2}\right)  $ are
non-degenerate in the sense that there are no multiple factors when they are
factored into linear terms over complex numbers. In this paper, we prove the
following theorems in more general situations.%

\begin{theorem}%
Let $S\left(  x_{1},x_{2},y\right)  =y^{n_{1}}h_{n-n_{1}}\left(  x_{1}%
,x_{2}\right)  +\cdots+y^{n_{s}}h_{n-n_{s}}\left(  x_{1},x_{2}\right)  ,$
where $h_{m}\left(  x_{1},x_{2}\right)  $ is a homogeneous polynomial of
degree $m$ in $x_{1}$ and $x_{2}$, $0<n_{1}<n_{2}<\cdots<n_{s}<n$, and both
$h_{n-n_{1}}\left(  x_{1},x_{2}\right)  $ and $h_{n-n_{s}}\left(  x_{1}%
,x_{2}\right)  $ are non-degenerate. Then, for the (2+1) dimensional
oscillatory integral operator defined in (\ref{ggg}) we have
\begin{equation}
\left \Vert T_{\lambda}\right \Vert _{L^{2}\left(  \mathbb{R}^{2}\right)
\rightarrow L^{2}\left(  \mathbb{R}\right)  }\leq \left \{
\begin{array}
[l]{l}%
C\lambda^{-1/\left(  2\delta \right)  }\text{, if }\delta>1,\\
C\lambda^{-1/2}\log \left(  \lambda \right)  \text{, if }\delta=1,
\end{array}
\right.  \text{ }\nonumber
\end{equation}
where $\delta=\max \! \left(  \frac{n}{3},\frac{n-n_{s}}{2},n_{1}\right)  $ and
$C$ is a constant independent of $\lambda$.%
\end{theorem}%

We further prove that the decay rate $1/\left(  2\delta \right)  $ is optimal,
and it is given by the modified Newton distance of the phase function, thus
verifies the conjecture of \citet{Greenleaf07} in this case.%

\begin{theorem}%
For the oscillatory integral operators considered in Theorem 1, if
$\Phi(0,0,0)\neq0$ then we have%
\begin{equation}
\left \Vert T_{\lambda}\right \Vert _{L^{2}\left(  \mathbb{R}^{2}\right)
\rightarrow L^{2}\left(  \mathbb{R}\right)  }\geq C\lambda^{-1/\left(
2\delta \right)  },
\end{equation}
where $C$ is a constant independent of $\lambda$. Thus, the decay rate
$1/\left(  2\delta \right)  $ is optimal. Further, $\delta$ is the modified
Newton distance of the phase function $S\left(  x_{1},x_{2},y\right)  $, i.e.,
the maximum Newton distance under affine transforms in $x_{1}$ and $x_{2}$.%
\end{theorem}%

Note that if $n_{1}\leq \frac{n}{3}$ and $n_{s}\geq \frac{n}{3}$, then
$\delta=\frac{n}{3}.$ Hence, our theorems include the result of \citet{tang06}
as a special case.

Since we only consider the $L^{2}$ norm, we use $\| \cdot \|$ to denote the
$L^{2}$ operator norm throughout the remainder of the paper. In addition, we
adopt the notation $\lesssim$ and $\gtrsim$ to simplify expressions: $a
\lesssim b$ (respectively, $a \gtrsim b$) means that $|a| \le C |b|$
(respectively, $|a| \ge C |b|$), where $C$ is a constant independent of
$\lambda$.

\section{Proof of the decay rate}

Similar to \citet{tang06}, we establish the decay rate by decomposing the
operator according to the size of the Hessian, following the general strategy
of \citet{Phong94}. The key idea in \citet{Phong94} is to partition the
integral region of the operator~(\ref{ggg}) based on the magnitude of the
mixed Hessian $S_{x,y}^{\prime \prime}$. In general, regions closer to the
singularity correspond to smaller scales. On each component of this
decomposition, the Van der Corput lemma can be applied, together with Young's
inequality, yielding the overall decay rate for (\ref{ggg}).

Since the amplitude $\Phi(x_{1},x_{2},y)$ has compact support, we may assume
without loss of generality that $\mathrm{supp}(\Phi)$ is contained in the unit
ball $|(x_{1},x_{2},y)|\leq1$. Rather than directly partitioning the entire
support, we first perform a conic decomposition to exploit the homogeneity of
the phase function. In particular, it suffices to prove the desired decay rate
within sufficiently small conic neighborhoods around points on the unit sphere
in $\mathbb{R}_{x,y}^{3}$. More precisely, given any open cover of the unit
sphere in $\mathbb{R}_{x,y}^{3}$, we may extract a finite subcover $\{
\sigma_{k}\}_{k=1}^{m}$ together with a corresponding partition of unity
$\sum_{k=0}^{m}\varphi_{k}(x_{1},x_{2},y)\equiv1$, where each $\varphi_{k}$ is
supported in $\sigma_{k}$. Then the operator $T_{\lambda}$ can be decomposed
as
\[
T_{\lambda}=\sum_{k=0}^{m}T_{\lambda}^{k},
\]
where $T_{\lambda}^{k}$ denotes the operator with the additional cutoff
$\varphi_{k}\! \left(  \frac{(x_{1},x_{2},y)}{|(x_{1},x_{2},y)|}\right)  $
inserted into the amplitude, namely
\begin{equation}
T_{\lambda}^{k}f(y)=\int e^{i\lambda S(x_{1},x_{2},y)}\Phi(x_{1}%
,x_{2},y)\varphi_{k}\left(  \frac{1}{|(x_{1},x_{2},y)|}(x_{1},x_{2},y)\right)
f(x_{1},x_{2})\,dx_{1}dx_{2}. \label{oioc}%
\end{equation}

Within each conic neighborhood, we further partition according to the distance
to the origin. The insertion of $\varphi_{k}\! \left(  \frac{(x_{1},x_{2}%
,y)}{|(x_{1},x_{2},y)|}\right)  $ does not affect the applicability of the Van
der Corput lemma, as it satisfies condition~(\ref{zz}) below. For notational
simplicity, we will continue to denote each such localized operator by
$T_{\lambda}$.

Fix a point $(a,b,c)$ on the unit sphere. Since at least one of $a,b,c$ is
nonzero, we may assume without loss of generality that $c\neq0$. For any
$\varepsilon>0$, define
\begin{equation}
R_{k}=\Big \{ \left(  x_{1},\;x_{2}\right)  \in \frac{1}{c}\left(  a,b\right)
y+2^{-k-1}\left(  -\varepsilon,\varepsilon \right)  \times \left(
-\varepsilon,\varepsilon \right)  ,\;y\in c(2^{-k-1},2^{-k+1})\Big \},
\label{r}%
\end{equation}
for $k=1,2,\dots$. Then $\{R_{k}\}_{k=1}^{\infty}$ forms a conic neighborhood
of $(a,b,c)$ within $\mathrm{supp}(\Phi)$. Let $\{ \varphi_{k}\}_{k=0}%
^{\infty}$ with $0\leq \varphi_{k}\leq1$ and $\sum_{k=0}^{\infty}\varphi
_{k}\equiv1$ be a corresponding partition of unity such that
\begin{equation}
\operatorname{supp}(\varphi_{k})\subset R_{k},\  \Vert \partial_{x}^{\alpha
}\partial_{y}^{\beta}\varphi_{k}\Vert_{\infty}\leq C_{\alpha \beta}%
2^{(|\alpha|+|\beta|)k}. \label{zz}%
\end{equation}

With this partition, we decompose the operator as
\[
T_{\lambda}=\sum_{k=0}^{\infty}T_{\lambda}^{k},
\]
where
\[
T_{\lambda}^{k}f(y)=\int e^{i\lambda S(x_{1},x_{2},y)}\Phi(x_{1},x_{2},y)\,
\varphi_{k}(x_{1},x_{2},y)\,f(x_{1},x_{2})\,dx_{1}dx_{2}.
\]
Condition (\ref{zz}) ensures that after inserting $\varphi_{k},$ the new
amplitude function $\Phi(x_{1},x_{2},y)\, \varphi_{k}(x_{1},x_{2},y)$ still
satisfies the condition for the amplitude function in the Van der Corput lemma
(condition (1) in Lemma 3 of \citet{tang06}). Further refinements of the
decomposition may be required depending on the nature of the singularity at
the point under consideration. In all cases, however, the decomposition is
carried out via a partition of the integration region, with each component
corresponding to an integral operator obtained by inserting a partition of
unity subordinate to this decomposition. The partition functions are chosen to
satisfy conditions analogous to (\ref{zz}), ensuring that the Van der Corput
lemma remains applicable. For brevity, in the remainder of the paper we
describe only the geometric decomposition of the region, with the
understanding that the corresponding operator decomposition follows from the
associated partition of unity.

\subsection{Nonsingular points}

At a non-singular point $\left(  a,b,c\right)  $, we have $\frac{\partial
^{2}S}{\partial x_{1}\partial y}\left(  a,b,c\right)  \neq0$ or $\frac
{\partial^{2}S}{\partial x_{2}\partial y}\left(  a,b,c\right)  \neq0.$ The
decay rate $\lambda^{-3/\left(  2n\right)  }$ $\left(  \lambda^{-1/2}%
\log \left(  \lambda \right)  \text{ if }n=3\right)  $ can be obtained following
the same approach as described in \citet{tang06}, by dyadically decomposing
the operator according to the distance to the origin using (\ref{r}).%

\begin{lemma}%
In a small conic neighborhood of a non-singular point $\left(  a,b,c\right)
$, we have $\left \Vert T_{\lambda}\right \Vert \lesssim \lambda^{-3/\left(
2n\right)  }$ or $\left(  \lambda^{-1/2}\log \left(  \lambda \right)  \right)  $
if $n=3$.%
\end{lemma}%
%

\begin{proof}%
We dyadically decompose the operator according to the distance to the origin
using (\ref{r}). For each of $T_{k}^{\lambda}$, we may apply the Van der
Corput lemma and Young's inequality to obtain two estimates according to the
size of the support of the operator as well as the magnitude of the Hessian.
More precisely, if $\frac{\partial^{2}S}{\partial x_{1}\partial y}\left(
x,y\right)  \neq0$, then $\frac{\partial^{2}S}{\partial x_{1}\partial
y}\left(  x,y\right)  \sim2^{-\left(  n-2\right)  k}$, by applying the Van der
Corput lemma (the oscillatory estimate Lemma 3 in \citet{tang06}) in the
($x_{1},y$) variables, and Young's inequality in $x_{2}$, we obtain
\begin{equation}
\left \Vert T_{k}^{\lambda}\right \Vert \lesssim \lambda^{-1/2}2^{(n-3)k/2}.
\label{non}%
\end{equation}
If $\frac{\partial^{2}S}{\partial x_{2}\partial y}\left(  x,y\right)  \neq0$,
then $\frac{\partial^{2}S}{\partial x_{2}\partial y}\left(  x,y\right)
\sim2^{-\left(  n-2\right)  k}$. Applying the Van der Corput lemma in the
($x_{2},y$) variables, and Young's inequality in $x_{1}$, we obtain the same
oscillatory estimate as (\ref{non}).

On the other hand, applying Young's inequality (the size estimate Lemma 4 in
\citet{tang06}) in all variables yields,
\[
\left \Vert T_{k}^{\lambda}\right \Vert \lesssim2^{-3k/2}.
\]

It follows that
\begin{align*}
\left \Vert T_{\lambda}\right \Vert  &  \lesssim \sum \nolimits_{k}\left \Vert
T_{k}^{\lambda}\right \Vert \lesssim \sum \nolimits_{k}\min \left(  \lambda
^{-1/2}2^{(n-3)k/2},2^{-3k/2}\right) \\
&  \lesssim \lambda^{-3/\left(  2n\right)  }\  \text{(or }\lambda^{-1/2}%
\log \left(  \lambda \right)  \text{ if }n=3\text{).}%
\end{align*}%
\end{proof}%

Next, we turn to the analysis of singular points. Let $(a,b,c)$ be a singular
point on the unit sphere. Then
\[
\frac{\partial^{2}S}{\partial x_{1}\partial y}(a,b,c)=\frac{\partial^{2}%
S}{\partial x_{2}\partial y}(a,b,c)=0.
\]
We classify the singular points into the following three cases:

\begin{enumerate}
\item Singular points on the $X$-plane: $c=0$;

\item Singular points on the $Y$-axis: $(a,b)=(0,0)$;

\item Singular points in general position: $c\neq0$ and $(a,b)\neq(0,0)$.
\end{enumerate}

\subsection{Singular points on the X-plane}

For any point $\left(  a,b,0\right)  $ with $\sqrt{a^{2}+b^{2}}=1,$ because of
the non-degeneracy of $h_{n-n_{1}}\left(  x_{1},x_{2}\right)  $, we have
$\left(  \frac{\partial}{\partial x_{1}}h_{n-n_{1}}\left(  a,b\right)
,\frac{\partial}{\partial x_{2}}h_{n-n_{1}}\left(  a,b\right)  \right)
\neq \left(  0,0\right)  $. We have the following lemma.%

\begin{lemma}%
In a small conic neighborhood of the point $\left(  a,b,0\right)  $, we have
$\left \Vert T_{\lambda}\right \Vert \lesssim \lambda^{-1/\left(  2\max \left(
n_{1},\frac{n-n_{1}}{2}\right)  \right)  }$ or $\lambda^{-1/2}\log \left(
\lambda \right)  $ if $\max \left(  n_{1},\frac{n-n_{1}}{2}\right)  =1$.%
\end{lemma}%
%

\begin{proof}%
Without loss of generality we assume $a\neq0$, then we decompose the operator
dyadically according to $x_{1}\sim2^{-i}$ and $y\sim2^{-k}$. Thus,
\[
T_{\lambda}=\sum \nolimits_{ik}T_{ik},
\]
where $T_{ik}$ is supported over region $x_{1}\sim a2^{-i},x_{2}\sim
2^{-i}\left(  b-\varepsilon,b+\varepsilon \right)  ,y\sim2^{-k}.$ If
$\frac{\partial}{\partial x_{1}}h_{n-n_{1}}\left(  a,b\right)  \neq0$, then
$\frac{\partial}{\partial x_{1}}h_{n-n_{1}}\left(  x_{1},x_{2}\right)
\sim2^{-i\left(  n-n_{1}\right)  }.$ For any constant $C$, there is a small
enough conic neighborhood of the point such that $k\geq i+C$. Hence
\[
\frac{\partial^{2}S}{\partial x_{1}\partial y}\left(  x,y\right)
\sim2^{-k\left(  n_{1}-1\right)  -i\left(  n-n_{1}-1\right)  }%
\]
over the support of $T_{ik}$. Therefore, using the operator Van der Corput
lemma in the $(x_{1},y)$ variables and Young's inequality in $x_{2}$, we
obtain
\begin{equation}
\left \Vert T_{ik}\right \Vert \lesssim \left(  \lambda2^{-k\left(
n_{1}-1\right)  -i\left(  n-n_{1}-1\right)  }\right)  ^{-1/2}2^{-i/2}%
=\lambda^{-1/2}2^{k\left(  n_{1}-1\right)  /2+i\left(  n-n_{1}-2\right)  /2}
\label{a}%
\end{equation}
If $\frac{\partial}{\partial x_{2}}h_{n-n_{1}}\left(  a,b\right)  \neq0$,
then
\begin{align*}
\frac{\partial^{2}S}{\partial x_{2}\partial y}\left(  x,y\right)   &
=n_{1}y^{n_{1}-1}\left[  \frac{\partial}{\partial x_{2}}h_{n-n_{1}}\left(
a,b\right)  +\cdots \right] \\
&  \sim2^{-k\left(  n_{1}-1\right)  -i\left(  n-n_{1}-1\right)  }.
\end{align*}
Thus we obtain the same oscillatory estimate as (\ref{a}).

On the other hand, Young's inequality in all variables yields the following
estimate based on the size of the support:
\begin{equation}
\left \Vert T_{ik}\right \Vert \lesssim2^{-(2i+k)/2}. \label{b}%
\end{equation}
Summing over $2i+k=d,$ we obtain%
\begin{equation}
\left \Vert T_{d}\right \Vert =\left \Vert \sum \nolimits_{2i+k=d}T_{ik}%
\right \Vert \lesssim2^{-d/2}\nonumber
\end{equation}
based on (\ref{b}) and
\begin{align}
\left \Vert T_{d}\right \Vert  &  =\left \Vert \sum \nolimits_{2i+k=d}%
T_{ik}\right \Vert \lesssim \lambda^{-1/2}2^{k\left(  n_{1}-1\right)
/2+i\left(  n-n_{1}-2\right)  /2}\label{d}\\
&  \lesssim \lambda^{-1/2}2^{\left(  \left(  2i+k\right)  \max \left(
n_{1}-1,\frac{n-n_{1}-2}{2}\right)  \right)  /2}\lesssim \lambda^{-1/2}%
2^{d\left(  \max \left(  n_{1},\frac{n-n_{1}}{2}\right)  -1\right)
/2}\nonumber
\end{align}
based on (\ref{a}). Note that different components with the same $2i+k=d$ are
almost orthogonal as both $i$ and $k$ are different, thus the support in $x$
and $y$ are almost disjoint.

It follows that the norm of $T_{\lambda}=\sum \nolimits_{d}T_{d},$%
\begin{align*}
\left \Vert T_{\lambda}\right \Vert  &  \leq \sum \nolimits_{d}\left \Vert
T_{d}\right \Vert \lesssim \sum \nolimits_{d}\min \left(  \lambda^{-1/2}%
2^{d\left(  \max \left(  n_{1},\frac{n-n_{1}}{2}\right)  -1\right)
/2},2^{-d/2}\right) \\
&  \lesssim \lambda^{-1/\left(  2\max \left(  n_{1},\frac{n-n_{1}}{2}\right)
\right)  }\text{ (}\lambda^{-1/2}\log \left(  \lambda \right)  \text{ if }%
\max \left(  n_{1},\frac{n-n_{1}}{2}\right)  =1\text{).}%
\end{align*}
Note that $\max \left(  n_{1},\frac{n-n_{1}}{2}\right)  =1$ happens only when
$n=3$ and $n_{1}=1.$%
\end{proof}%

\subsection{Singular points on the Y-axis}

We consider the singular point $\left(  0,0,1\right)  ;$ proof for the other
singular point on the unit sphere, $\left(  0,0,-1\right)  ,$ can be carried
out similarly. We prove the following lemma.%

\begin{lemma}%
In a small conic neighborhood of the point $\left(  0,0,1\right)  $, we have
$\left \Vert T_{\lambda}\right \Vert \lesssim \lambda^{-1/\left(  2\max \left(
\frac{n}{3},\frac{n-n_{s}}{2}\right)  \right)  }$ or $\lambda^{-1/2}%
\log \left(  \lambda \right)  $ if $\max \left(  \frac{n}{3},\frac{n-n_{s}}%
{2}\right)  =1$.%
\end{lemma}%
%

\begin{proof}%
For a small enough conic neighborhood of the points on the Y-axis $\left(
0,0,1\right)  $, we further decompose the X-plane conically. So, it is
sufficient if we can show that for each point $\left(  a,b\right)  $ on the
unit circle of the X-plane, an operator supported in a small enough convex
conic neighborhood of the segment from $\left(  0,0,1\right)  $ to $\left(
a,b,1\right)  $ has the desired decay rate.

Without loss of generality, we assume that $a\neq0$, then we decompose the
operator dyadically according to $x_{1}\sim a2^{-i}$ and $y\sim2^{-k}$. Thus,
$T_{ik}$ is supported in
\[
x_{1}\sim a2^{-i}\left(  1-\varepsilon,1+\varepsilon \right)  ,x_{2}\sim
2^{-i}\left(  b-\varepsilon,b+\varepsilon \right)  ,y\sim2^{-k}.
\]
For any constant $C$, there is a small enough conic neighborhood of the point
such that $i\geq k+C$. Because of the Non-degeneracy of $h_{n-n_{s}}\left(
x_{1},x_{2}\right)  $, we have
\[
\left(  \frac{\partial}{\partial x_{1}}h_{n-n_{s}}\left(  a,b\right)
,\frac{\partial}{\partial x_{2}}h_{n-n_{s}}\left(  a,b\right)  \right)
\neq \left(  0,0\right)  .
\]
If $\frac{\partial}{\partial x_{1}}h_{n-n_{s}}\left(  a,b\right)  \neq0$, then
$\frac{\partial}{\partial x_{1}}h_{n-n_{s}}\left(  \frac{x_{1}}{y},\frac
{x_{2}}{y}\right)  \sim2^{-i\left(  n-n_{s}-1\right)  \left(  k-i\right)  }$
over the support of $T_{ik}$. It follws that
\[
S\left(  \frac{x_{1}}{y},\frac{x_{2}}{y},1\right)  =y^{n}\left(  h_{n-n_{1}%
}\left(  \frac{x_{1}}{y},\frac{x_{2}}{y}\right)  +\cdots+h_{n-n_{s}}\left(
\frac{x_{1}}{y},\frac{x_{2}}{y}\right)  \right)  \sim2^{\left(  n-n_{s}%
-1\right)  \left(  k-i\right)  }%
\]
over the support of $T_{ik}$ as the remainder terms have higher orders in
$\frac{x_{1}}{y}$ and$\frac{x_{2}}{y}$. Thus,
\[
S\left(  x_{1},x_{2},y\right)  \sim2^{-k\left(  n_{s}-1\right)  -i\left(
n-n_{s}-1\right)  }.
\]
Therefore, using the operator Van der Corput lemma in the $(x_{1},y)$
variables, and Young's inequality in $x_{2}$, we obtain
\begin{align}
\left \Vert T_{ik}\right \Vert  &  \lesssim \left(  \lambda2^{-k\left(
n_{1}-1\right)  -i\left(  n-n_{s}-1\right)  }\right)  ^{-1/2}2^{-i/2}%
\label{e}\\
&  =\lambda^{-1/2}2^{\left(  k\left(  n_{s}-1\right)  +i\left(  n-n_{s}%
-2\right)  \right)  /2}\nonumber
\end{align}
If $\frac{\partial}{\partial x_{2}}h_{n-n_{s}}\left(  a,b\right)  \neq0$,
then
\[
\frac{\partial^{2}S}{\partial x_{2}\partial y}\left(  x,y\right)  \sim
\lambda^{-1/2}2^{k\left(  n_{s}-1\right)  +i\left(  n-n_{s}-1\right)  }.
\]
Using the operator Van der Corput lemma in the $(x_{2},y)$ variables, and
Young's inequality in $x_{1}$, we obtain the same oscillatory estimate as
(\ref{e}).

On the other hand, Young's inequality in all variables yields the following
estimate based on the size of the support:
\[
\left \Vert T_{ik}\right \Vert \lesssim2^{-(2i+k)/2}.
\]
Summing over $2i+k=d,$ we obtain%
\begin{equation}
\left \Vert T_{d}\right \Vert =\left \Vert \sum \nolimits_{2i+k=d}T_{ik}%
\right \Vert \lesssim2^{-d/2}\nonumber
\end{equation}
and
\begin{align}
\left \Vert T_{d}\right \Vert  &  =\left \Vert \sum \nolimits_{2i+k=d}%
T_{ik}\right \Vert \lesssim \lambda^{-1/2}2^{\left(  k\left(  n_{1}-1\right)
-i\right)  /2}\\
&  \lesssim \lambda^{-1/2}2^{\left(  \left(  2i+k\right)  \left(
n_{1}-1\right)  \right)  /2}\lesssim \lambda^{-1/2}2^{d\left(  n_{1}-1\right)
/2}\nonumber
\end{align}

Again different components with the same $2i+k=d$ are almost orthogonal as
both $i$ and $k$ are different, thus the support in $x_{1}$ and $y$ are disjoint.

It follows that
\begin{align*}
\left \Vert T_{\lambda}\right \Vert  &  \leq \sum \nolimits_{d}\left \Vert
T_{d}\right \Vert \lesssim \sum \nolimits_{d}\min \left(  \lambda^{-1/2}%
2^{d\left(  n_{1}-1\right)  /2},2^{-d/2}\right) \\
&  \lesssim \lambda^{-1/\max \left(  \frac{2n}{3},n-n_{s}\right)  }\text{ or
}(\lambda^{-1/2}\log \left(  \lambda \right)  \text{ if }\max \left(  \frac{n}%
{3},\frac{n-n_{s}}{2}\right)  =1).
\end{align*}
Note that $\max \left(  \frac{n}{3},\frac{n-n_{s}}{2}\right)  =1$ happens only
when $n=3$ and $n_{s}=1.$%

\end{proof}%

\subsection{Singular points in general position}%

\begin{lemma}%
In a small conic neighborhood of a (singular) point $\left(  a,b,c\right)  $
with $\left(  a,b\right)  \neq \left(  0,0\right)  $ and $c\neq0$,\textbf{
}$\left \Vert T_{\lambda}\right \Vert \lesssim \lambda^{-3/\left(  2n\right)  }$
or ($\lambda^{-1/2}\log \left(  \lambda \right)  $ if $n=3$)$.$%
\end{lemma}%
%

\begin{proof}%
By close inspection it is clear that the proof of \citet{tang06} in the cases
of singular points in general position only used the assumption that
$h_{n-n_{1}}\left(  x_{1},x_{2}\right)  $ and $h_{n-n_{s}}\left(  x_{1}%
,x_{2}\right)  $ are non-degenerate, and not rely on the assumption that
$n_{1}\leq \frac{n}{3}$ or $n_{s}\geq \frac{n}{3}$. Thus, this situation can be
proved exactly the same way as that in \citet{tang06}. This completes the
proof of Theorem 1.%

\end{proof}%

Note that throughout the proof, we actually obtain different decay rates for
the operator within small conic neighborhoods of different directions. Next,
we prove that the decay rates we obtained for each direction is sharp, thus
the overall decay rate in Theorem 1 is also sharp.

\section{Sharpness of the decay rate}

We use test functions to show that the decay rates given in the last section
are optimal if $\Phi(0,0,0)\neq0$. Without loss of generality we can assume
that $\Phi(x_{1},x_{2},y)>1$ on a neighborhood of the origin.

For any point $\left(  a,b,c\right)  $, we can simply use the test function%
\begin{equation}
f_{\lambda}(x_{1},x_{2})=%
\begin{cases}
1, & x_{1}\in \lambda^{-1/n}\left(  -\varepsilon,\varepsilon \right)  ,x_{2}%
\in \lambda^{-1/n}\left(  1-\varepsilon,1+\varepsilon \right)  ,\\
0, & \text{otherwise,}%
\end{cases}
\label{t1}%
\end{equation}
to get the lower bound of $\lambda^{-3/2n}.$ Here $\varepsilon$ is a small
positive number such that $|S\left(  x_{1},x_{2},y\right)  -S\left(
a,b,c\right)  |<\eta$ for a small enough positive number $\eta$ for all
$x_{1}\in \left(  a-\varepsilon,a+\varepsilon \right)  $, $x_{2}\in \left(
b-\varepsilon,b+\varepsilon \right)  ,$ and $y\in \left(  c-\varepsilon
,c+\varepsilon \right)  $. Then,
\begin{equation}
\left \Vert f_{\lambda}\right \Vert \leq \lambda^{-1/n}.\label{gg}%
\end{equation}
and for any point $y\in \lambda^{-1/n}\left(  1-\varepsilon,1+\varepsilon
\right)  $ , we have $|\lambda S\left(  x_{1},x_{2},y\right)  -S\left(
a,b,c\right)  |=|S\left(  \lambda^{1/n}x_{1},\lambda^{1/n}x_{2},\lambda
^{1/n}y\right)  -S\left(  a,b,c\right)  |<\eta \ $over the rectangle $x_{1}%
\in \lambda^{-1/n}\left(  -\varepsilon,\varepsilon \right)  ,$ $x_{2}\in
\lambda^{-1/n}\left(  1-\varepsilon,1+\varepsilon \right)  .$ Thus, for a small
enough $\eta$ we have either the real or the imaginary part of $e^{i\lambda
S\left(  x_{1},x_{2},y\right)  }>c$ , a positive constant not depending on
$\lambda.$ So, for $y\in \lambda^{-1/n}\left(  1-\varepsilon,1+\varepsilon
\right)  $ we have
\begin{align*}
|T_{\lambda}f_{\lambda}\left(  y\right)  | &  =|\iint e^{i\lambda S}\Phi
(x_{1},x_{2},y)f_{\lambda}(x_{1},x_{2})dx_{1}dx_{2}|\\
&  \geq c\int_{x_{1}\in \lambda^{-1/n}\left(  -\varepsilon,\varepsilon \right)
,x_{2}\in \lambda^{-1/n}\left(  1-\varepsilon/2,1+\varepsilon/2\right)
}f_{\lambda}(x_{1},x_{2})dx_{1}dx_{2}|\\
&  \gtrsim \lambda^{-2/n}.
\end{align*}
Thus,$\ $
\[
\left \Vert T_{\lambda}f_{\lambda}\right \Vert ^{2}\geq \int_{y\in \lambda
^{-1/n}\left(  1-\varepsilon,1+\varepsilon \right)  }T_{\lambda}f_{\lambda
}\left(  y\right)  ^{2}dy\gtrsim \lambda^{-5/n}.
\]
and%
\begin{equation}
\left \Vert T_{\lambda}f_{\lambda}\right \Vert \gtrsim \lambda^{-5/2n}.\label{ff}%
\end{equation}
By (\ref{gg}) and (\ref{ff}), we have
\[
\left \Vert T_{\lambda}\right \Vert \geq \frac{\left \Vert T_{\lambda}f_{\lambda
}\right \Vert }{\left \Vert f_{\lambda}\right \Vert }\gtrsim \frac{\lambda
^{-5/2n}}{\lambda^{-1/n}}=\lambda^{-3/2n}.
\]
This shows that for nonsingular points the optimal decay rates are exactly
$\lambda^{-3/2n}$, except for a factor of log when $n=3.$

\subsection{Singular points on the X-plane}

For any singular point $\left(  a,b,0\right)  $ on the X-plane, we need only
to consider the case $\max \left(  n_{1},\frac{n}{3}\right)  =n_{1}$, as the
case $\max \left(  n_{1},\frac{n}{3}\right)  =\frac{n}{3}$ is covered by the
test function (\ref{t1}) above.

By the continuity of $y^{n_{1}}h_{n-n_{1}}\left(  x_{1},x_{2}\right)  $ at
$\left(  a,b,0\right)  $, for any $\eta>0$ there exists $\varepsilon>0$ such
that $|y^{n_{1}}h_{n-n_{1}}\left(  x_{1},x_{2}\right)  |<\frac{\eta}{2}$ for
all $x_{1}\in \left(  a-\varepsilon,a+\varepsilon \right)  ,x_{2}\in \left(
b-\varepsilon,b+\varepsilon \right)  ,$ and $y\in \left(  -\varepsilon
,+\varepsilon \right)  $. We choose the following test function%
\[
f(x_{1},x_{2})=%
\begin{cases}
1, & x_{1}\in \left(  a-\varepsilon,a+\varepsilon \right)  ,x_{2}\in \left(
b-\varepsilon,b+\varepsilon \right)  \\
0, & \text{otherwise.}%
\end{cases}
\]
Since the test function does not change with $\lambda$, its norm is a
constant. For any point $y$ with $\lambda^{1/n_{1}}y\in \left(  -\varepsilon
,\varepsilon \right)  $ and any $x_{1}\in \left(  a-\varepsilon,a+\varepsilon
\right)  $ and $x_{2}\in \left(  b-\varepsilon,b+\varepsilon \right)  ,$ we have%

\[
|\lambda y^{n_{1}}h_{n-n_{1}}\left(  x_{1},x_{2}\right)  |<\frac{\eta}{2}.
\]
On the other hand, for any point $y\ $with $\lambda^{1/n_{1}}y\in \left(
-\varepsilon,\varepsilon \right)  $, we have $y^{m}\sim o\left(  \lambda
^{-1}\right)  $ if $m>n_{1},$ thus
\begin{align*}
&  |\lambda S\left(  x_{1},x_{2},y\right)  -\lambda y^{n_{1}}h_{n-n_{1}%
}\left(  x_{1},x_{2}\right)  |\\
&  =|\lambda \left(  y^{n_{2}}h_{n-n_{2}}\left(  x_{1},x_{2}\right)
+\cdots+y^{n_{s}}h_{n-n_{s}}\left(  x_{1},x_{2}\right)  \right)  |<\frac{\eta
}{2},
\end{align*}
for $\lambda$ large enough. It follows that
\begin{align*}
&  |\lambda S\left(  x_{1},x_{2},y\right)  |\leq|\lambda S\left(  x_{1}%
,x_{2},y\right)  -\lambda y^{n_{1}}h_{n-n_{1}}\left(  x_{1},x_{2}\right)  |\\
&  +|\lambda y^{n_{1}}h_{n-n_{1}}\left(  x_{1},x_{2}\right)  -c^{n_{1}%
}h_{n-n_{1}}\left(  a,b\right)  |<\eta
\end{align*}

Thus, the real part of $e^{i\lambda S\left(  x_{1},x_{2},y\right)  }%
>\cos \left(  \eta \right)  $ (if we choose any $\eta<1$), a positive constant
not depending on $\lambda.$ So, for $\lambda^{1/n_{1}}y\in \left(
-\varepsilon,\varepsilon \right)  $
\begin{align*}
|T_{\lambda}f\left(  y\right)  | &  =|\iint e^{i\lambda S}\Phi(x_{1}%
,x_{2},y)f(x_{1},x_{2})dx_{1}dx_{2}|\\
&  \geq|\cos \left(  \eta \right)  \int_{x_{1}\in \left(  a-\varepsilon
,a+\varepsilon \right)  ,x_{2}\in \left(  b-\varepsilon,b+\varepsilon \right)
}f(x_{1},x_{2})dx_{1}dx_{2}|\\
&  \gtrsim c.
\end{align*}
Thus,$\ $
\[
\left \Vert T_{\lambda}f\right \Vert ^{2}\geq \int \nolimits_{y\in \lambda
^{-1/n_{1}}\left(  1-\varepsilon,1+\varepsilon \right)  }|T_{\lambda}f\left(
y\right)  |^{2}dy\gtrsim \lambda^{-1/n_{1}}.
\]
and%
\begin{equation}
\left \Vert T_{\lambda}f\right \Vert \gtrsim \lambda^{-1/2n_{1}}.\label{f}%
\end{equation}
It follows that
\[
\left \Vert T_{\lambda}\right \Vert \geq \frac{\left \Vert T_{\lambda}f\right \Vert
}{\left \Vert f\right \Vert }\gtrsim \lambda^{-1/2n_{1}}.
\]
This proves the sharpness of the decay rate in (\ref{d}).

\subsection{Singular points on the Y-axis}

For the singular point $\left(  0,0,1\right)  ,$ we need only to consider the
case $\max \left(  \frac{n-n_{s}}{2},\frac{n}{3}\right)  =\frac{n-n_{s}}{2}$,
as the case $\max \left(  \frac{n-n_{s}}{2},\frac{n}{3}\right)  =\frac{n}{3}$
is again covered by the test function (\ref{t1}). Let
\[
f_{\lambda}(x_{1},x_{2})=%
\begin{cases}
1, & x_{1}\in \lambda^{-1/\left(  n-n_{s}\right)  }\left(  -\varepsilon
,\varepsilon \right)  ,x_{2}\in \lambda^{-1/\left(  n-n_{s}\right)  }\left(
-\varepsilon,\varepsilon \right)  \\
0, & \text{otherwise}.
\end{cases}
\]
where $\varepsilon$ is a small positive number such that $|S\left(
x_{1},x_{2},y\right)  |<\frac{\eta}{2}$ for a small enough positive number
$\eta$ for all $x_{1}\in \left(  -\varepsilon,\varepsilon \right)  ,x_{2}%
\in \left(  -\varepsilon,\varepsilon \right)  ,$ and $y\in \left(  1-\varepsilon
,1+\varepsilon \right)  $. Then,
\begin{equation}
\left \Vert f_{\lambda}\right \Vert ^{2}=\int f_{\lambda}(x_{1},x_{2})^{2}%
dx_{1}dx_{2}\sim \lambda^{-2/\left(  n-n_{s}\right)  }\text{ and }\left \Vert
f_{\lambda}\right \Vert _{L^{2}\left(  \mathbb{R}\right)  }\sim \lambda
^{-1/\left(  n-n_{s}\right)  }.\label{lll}%
\end{equation}
For any point $y\in \left(  1-\varepsilon,1+\varepsilon \right)  $, we have%

\[
|\lambda y^{n_{s}}h_{n-n_{s}}\left(  x_{1},x_{2}\right)  |<\frac{\eta}{2}.
\]
On the other hand, for any point $x_{1}\in \lambda^{-1/\left(  n-n_{s}\right)
}\left(  -\varepsilon,\varepsilon \right)  ,x_{2}\in \lambda^{-1/\left(
n-n_{s}\right)  }\left(  -\varepsilon,\varepsilon \right)  ,$ we have
$h_{m}\left(  x_{1},x_{2}\right)  \lambda \sim o\left(  \lambda^{-1}\right)  $
for any $m<n_{s},$ thus
\begin{align*}
&  |\lambda S\left(  x_{1},x_{2},y\right)  -\lambda y^{n_{s}}h_{n-n_{s}%
}\left(  x_{1},x_{2}\right)  |\\
&  =|\lambda \left(  y^{n_{1}}h_{n-n1}\left(  x_{1},x_{2}\right)
+\cdots+y^{n_{s}}h_{n-n_{s}}\left(  x_{1},x_{2}\right)  \right)  |<\frac{\eta
}{2}.
\end{align*}
It follows that
\[
|\lambda S\left(  x_{1},x_{2},y\right)  |<|\lambda y^{n_{s}}h_{n-n_{s}}\left(
x_{1},x_{2}\right)  |+|\lambda S\left(  x_{1},x_{2},y\right)  -\lambda
y^{n_{s}}h_{n-n_{s}}\left(  x_{1},x_{2}\right)  |<\eta
\]
over the rectangle $x_{1}\in \left(  1-\varepsilon/2,1+\varepsilon/2\right)
,x_{2}\in \left(  1-\varepsilon/2,1+\varepsilon/2\right)  .$ Thus, the real
part of $e^{i\lambda S\left(  x_{1},x_{2},y\right)  }>\cos \left(  \eta \right)
$, a positive constant not depending on $\lambda$, if we choose any $\eta<1.$
So, for $y\in \left(  1-\varepsilon,1+\varepsilon \right)  $ we have
\begin{align*}
|T_{\lambda}f_{\lambda}\left(  y\right)  | &  \geq \cos \left(  \eta \right)
\iint_{x_{1}\in \lambda^{-1/\left(  n-n_{s}\right)  }\left(  -\varepsilon
,\varepsilon \right)  ,x_{2}\in \lambda^{-1/\left(  n-n_{s}\right)  }\left(
-\varepsilon,\varepsilon \right)  }|f_{\lambda}(x)|dx_{1}dx_{2}\\
&  \gtrsim \lambda^{-2/\left(  n-n_{s}\right)  },
\end{align*}
and$\ $
\begin{equation}
\left \Vert T_{\lambda}f_{\lambda}\right \Vert \geq \left[  \int_{y\in \left(
1-\varepsilon,1+\varepsilon \right)  }|T_{\lambda}f_{\lambda}\left(  y\right)
|^{2}dy\right]  ^{1/2}\gtrsim \lambda^{-2/\left(  n-n_{s}\right)
}.\label{lll2}%
\end{equation}
By (\ref{lll}) and (\ref{lll2}), we have
\[
\left \Vert T_{\lambda}\right \Vert \geq \frac{\left \Vert T_{\lambda}f_{\lambda
}\right \Vert }{\left \Vert f_{\lambda}\right \Vert }\gtrsim \frac{\lambda
^{-2/\left(  n-n_{s}\right)  }}{\lambda^{-1/\left(  n-n_{s}\right)  }}%
=\lambda^{-1/\left(  n-n_{s}\right)  }.
\]
This completes the proof of the optimality part of Theorem 2.

\section{Relation with the Newton distance}

Since the phase function is a homogeneous polynomial of degree $n$, the three
coordinates of any convex combination of the vertices in the Newton polyhedron
always add up to $n.$ Thus, the Newton distance $\geq$ $\frac{n}{3}$. Next, we
prove the Newton distance according to which of the numbers $\left(  \frac
{n}{3},\frac{n-n_{s}}{2},n_{1}\right)  $ takes the maximum value.

\textbf{Case I. }$\delta=\frac{n}{3}$. In such cases, $n_{1}\leq \frac{n}{3}$
and $n_{s}\geq$ $\frac{n}{3}$. For a non-degenerate homogenous polynomial of
$x_{1}$ and $x_{2}$ with degree $m$, the corresponding Newton polygon of the
polynomial on the $\left(  x_{1},x_{2}\right)  $ plane contains both points
$\left(  1,m-1\right)  $ and $\left(  m-1,1\right)  $, and hence the point
$\left(  \frac{m}{2},\frac{m}{2}\right)  ,$ if $m>1.$ Thus, by the
non-degeneracy of the homogeneous polynomials $h_{n-n_{1}}$ and $h_{n-n_{s}},$
the Newton polyhedron of contains points
\begin{equation}
P_{1}=\left(  \frac{n-n_{1}}{2},\, \frac{n-n_{1}}{2},\,n_{1}\right)  ,\qquad
P_{2}=\left(  \frac{n-n_{s}}{2},\, \frac{n-n_{s}}{2},\,n_{s}\right)  ,
\label{pp}%
\end{equation}
if $n_{s}<n-1.$ Since%
\[
\frac{n_{s}-\frac{n}{3}}{n_{s}-n_{1}}\left(  \frac{n-n_{1}}{2},\,
\frac{n-n_{1}}{2},\,n_{1}\right)  +\frac{\frac{n}{3}-n_{1}}{n_{s}-n_{1}%
}\left(  \frac{n-n_{s}}{2},\, \frac{n-n_{s}}{2},\,n_{s}\right)  =\left(
\frac{n}{3},\frac{n}{3},\frac{n}{3}\right)  \text{,}%
\]
it follows that the Newton distance is $\frac{n}{3}$.

If $n_{s}=n-1$, then $h_{n-n_{1}}$ is a non-degenerate homogeneous polynomials
with degree larger than 1, the Newton polyhedron contains points
\[
P_{3}=\left(  1,\,n-n_{1}-1,\,n_{1}\right)  ,\qquad P_{4}=\left(
n-n_{1}-1,1,\,n_{1}\right)  .
\]
However, $h_{n-n_{s}}$ has only degree one, thus it may only contain one of
the two points $\left(  1,\,0,\,n-1\right)  $ and $\left(  0,\,1,\,n-1\right)
.$ However, the following convex combination%
\[
\frac{n-3n_{1}}{3(n-n_{1}-1)}P_{3}+\frac{n^{2}-nn_{1}-3n+3}{3(n-n_{1}%
-1)(n-n_{1}-2)}P_{4}+\frac{n-3}{3(n-n_{1}-2)}P_{5}=\left(  \frac{n}{3}%
,\frac{n}{3},\frac{n}{3}\right)  \text{,}%
\]
where $P_{5}=\left(  1,\,0,\,n-1\right)  ,$ shows that the Newton distance is
$\frac{n}{3}$.

\textbf{Case II. }$\delta=n_{1}.$ In such case $n_{1}\geq$ $\frac{n}{3},$ then
all the vertices of the Newton Polygon have the y-coordinate at least $n_{1}$,
it follows that the Newton distance $\geq$ $n_{1}$. On the other hand, note
that the point $P_{1}=\left(  \frac{n-n_{1}}{2},\, \frac{n-n_{1}}{2}%
,\,n_{1}\right)  $ is in the Newton polyhedron if $h_{n-n_{1}}$ has a degree
at least 2, i.e., $n_{1}<n-1$. If $n_{1}=n-1$, then one of the two points
$\left(  1,\,0,\,n_{1}\right)  $ and $\left(  0,\,1,\,n_{1}\right)  $ is in
the Newton polyhedron. Hence, the Newton distance is $n_{1}$.

\textbf{Case III. }$\delta=\frac{n-n_{s}}{2}.$ In such case $n_{s}\leq \frac
{n}{3}$, then all the vertices of the Newton Polygon have the total of the two
x-coordinates at least $n-n_{s}.$ Thus, the Newton distance will be at least
$\frac{n-n_{s}}{2}.$ On the other hand, since $h_{n-n_{s}}$ has a degree at
least two, the point $\left(  \frac{n-n_{s}}{2},\, \frac{n-n_{s}}{2}%
,\,n_{s}\right)  $ is in the Newton polyhedron. Since $n_{s}\leq \frac{n-n_{s}%
}{2}$ in this case, it follows $\frac{n-n_{s}}{2}$ is indeed the Newton
distance. This completes the proof that the Newton distance of the phase
function is $\delta.$

This shows that the Newton distance of the phase functions in Theorem 1 is
determined by $n,$ $n_{1}$, and $n_{s}$. Consequently, it is invariant under
affine transformations of $\left(  x_{1},x_{2}\right)  $, and the modified
Newton distance of the phase functions is also $\delta.$ This completes the
proof of Theorem 2.

\textbf{Acknowledgement.} We would like to thank Prof. A. Greenleaf at the
University of Rochester for carefully reading an earlier version of the
manuscript and for providing valuable suggestions that significantly improved it.

\end{document}